\documentclass{amsart}
\usepackage[utf8]{inputenc}
\usepackage[english]{babel}
\usepackage{amsmath}
\usepackage{amsthm}
\usepackage{amssymb}
\usepackage{color}
\usepackage{dsfont}
\usepackage{hyperref}

\theoremstyle{plain}
\newtheorem{theorem}{Theorem}[section]
\newtheorem{lemma}[theorem]{Lemma}
\newtheorem{corollary}[theorem]{Corollary}
\newtheorem{propn}[theorem]{Proposition}
\newtheorem{conj}[theorem]{Conjecture}

\theoremstyle{definition}

\newtheorem{eg}[theorem]{Example}

\newcommand*{\field}[1]{\mathbb{#1}}

\title{Newform Eisenstein congruences of local origin}
\author{Dan Fretwell, Jenny Roberts}
\date{}

\begin{document}

\maketitle

\begin{center}\textit{In memory of Lynne Walling.}\end{center}

\begin{abstract}
We give a general conjecture concerning the existence of Eisenstein congruences between weight $k\geq 3$ newforms of square-free level $NM$ and weight $k$ new Eisenstein series of square-free level $N$. Our conjecture allows the forms to have arbitrary character $\chi$ of conductor $N$. The special cases $M=1$ and $M=p$ prime are fully proved, with partial results given in general. We also consider the relation with the Bloch-Kato conjecture, and finish with computational examples demonstrating cases of our conjecture that have resisted proof.
\end{abstract}

\section{Introduction}

The theory of Eisenstein congruences has a rich and beautiful history, beginning with Ramanujan's remarkable observation that the Fourier coefficients $\tau(n)$ of the discriminant function: \[\Delta(z) = q\prod_n(1-q^n)^{24} = \sum_{n\geq 1}\tau(n)q^n \in S_{12}(\text{SL}_2(\mathbb{Z}))\] satisfy $\tau(n)\equiv \sigma_{11}(n) \bmod 691$ for all $n\geq 1$ (here $\sigma_{11}(n) = \sum_{d\mid n}d^{11}$). Intuitively, this family of congruences is explained via a congruence between two modular forms of weight $12$; the cusp form $\Delta$ and the Eisenstein series $E_{12}$. The significance of the modulus $691$ is that it divides the numerator of $-\frac{B_{12}}{24}$, the constant term of $E_{12}$ (so that $E_{12}$ is a cusp form mod $691$).

Since Ramanujan, many other congruences have been found between cusp forms and Eisenstein series, modulo other interesting primes. For example, one can vary the weight of the forms and find congruences whose moduli divide the numerators of other Bernoulli numbers. The existence of such congruences has been a key tool in the proofs of various important results in algebraic number theory, e.g. the Herbrand-Ribet Theorem, relating the $p$-divisibility of Bernoulli numbers with the Galois module structure of $\text{Cl}(\mathbb{Q}(\zeta_p))[p]$. (see \cite{ribet_1976}).

Similarly, varying the levels and characters of our forms produces even more congruences. This time the moduli are observed to divide numerators of generalised Bernoulli numbers and special values of local Euler factors of Dirichlet $L$-functions. In the latter case such congruences are often referred to as being of ``local origin". The papers, \cite{mazur},\cite{dummigan_2007}, \cite{dummigan_fretwell_2014}, \cite{billerey_menares_2016}, \cite{billerey_menares_2017} and \cite{spencer_2018} contain thorough discussions of such congruences. The existence of these congruences is linked to special cases of the Bloch-Kato Conjecture, a far reaching generalisation of the Herbrand-Ribet Theorem, the Analytic Class Number Formula and the Birch Swinnerton-Dyer Conjecture. This conjecture implies links between $p$-divisibility of special values of certain motivic $L$-functions and $p$-torsion in certain Bloch-Kato Selmer groups.

More generally, the theory of Eisenstein congruences has been extended to various families of automorphic form, although the landscape is still highly conjectural \cite{harder}. Roughly speaking, if $G/\mathbb{Q}$ is a reductive group then one expects to observe congruences between Hecke eigenvalues coming from cuspidal automorphic representations for $G(\mathbb{A}_{\mathbb{Q}})$ and automorphic representations that are parabolically induced from Levi subgroups of $G$. The moduli of such congruences are also predicted to arise from special values of certain motivic $L$-functions (related to the particular Levi subgroup considered), in direct comparison with the Bloch-Kato conjecture. For detailed discussion of results and conjectures in this direction, see \cite{bergstrom}. In the special case of $\text{GL}_2/\mathbb{Q}$ we recover exactly the Eisenstein congruences mentioned above, the relevant $L$-functions being Dirichlet $L$-functions (possibly incomplete, hence the appearance of local Euler factors). It is expected that proving such congruences will provide key insights into high rank cases of the Bloch-Kato Conjecture.

In this paper we will return to the ``classical" case of $\text{GL}_2$ Eisenstein congruences, but focus instead on the existence of newforms that satisfy such congruences. Progress has been made on this question in the case of trivial character (\cite{billerey_menares_2016}, \cite{dummigan_fretwell_2014}), and so we consider the more general case of forms with arbitrary character of square-free conductor. To do this, we discuss the following general conjecture and provide a full proof in the special cases of $M=1$ and $M=p$ prime. 

\begin{conj}
Let $N,M\geq 1$ be square-free, $k > 2$ and $l>k+1$ be a prime satisfying $l \nmid \varphi(N)NM$. Let $\psi,\phi$ be Dirichlet characters of conductors $u,v\geq 1$, satisfying $N = uv$, and set $\chi = \psi\phi$ (with lift $\tilde{\chi}$ to modulus $NM$). There exists a newform $f \in S_k^{\text{new}}(\Gamma_0(NM), \tilde{\chi})$ and a prime $\lambda\mid l$ of $\mathcal{O}_f[\phi,\psi]$ such that 
\begin{equation*}
    a_q(f) \equiv \psi(q)+\phi(q)q^{k-1} \ \text{mod } \lambda
\end{equation*}
for all primes $q\nmid NM$ if and only if both of the following conditions hold for some $\lambda'|l$ in $\mathbb{Z}[\psi,\phi]$ (satisfying $\lambda|\lambda'$):
\begin{enumerate}
\item $\text{ord}_{\lambda'}(L(1-k, \psi^{-1}\phi) \prod_{p \in \mathcal{P}_M} (\psi(p) - \phi(p)p^k)) > 0$.
\item $\text{ord}_{\lambda'}((\psi(p)-\phi(p)p^k)(\psi(p)-\phi(p)p^{k-2}))>0$ for each prime $p\in \mathcal{P}_M$.
\end{enumerate}
\end{conj} 

\noindent Here, $\mathcal{P}_M$ is the set of primes divisors of $M$, $a_n(f)$ is the $n$th Fourier coefficient of $f$ and $\mathcal{O}_f[\psi,\phi]$ is the ring of integers of the smallest extension of $K_f = \mathbb{Q}(\{a_n(f)\})$ containing the values of $\psi$ and $\phi$ (similarly for $\mathbb{Z}[\psi,\phi]$).

After, we consider the natural relationship between these newform congruences and the Bloch-Kato conjecture, and give computational evidence for our conjecture in cases where it is not known.

\textbf{Acknowledgements} We thank Neil Dummigan for useful discussions concerning links between Conjecture $1.1$ and the Bloch-Kato conjecture, and for helpful comments and suggestions for improvement. We also thank an anonymous referee for valuable comments.

This work forms part of the thesis of the second named author, and we are grateful to the Heilbronn Institute for support via their PhD Studentship program. 

Finally, we wish to dedicate this work to the memory of Lynne Walling, a close friend, mentor and collaborator of the first author, and who would have been a PhD supervisor of the second. Lynne was a highly valued member of the Mathematics community, whose encouragement and support towards early career mathematicians and those from under-represented groups was second to none.

\section{Background and notation}
\subsection{The Setup}
We recap the background knowledge of modular forms that we will need, and refer the reader to \cite{diamond_shurman_2016} for further definitions and discussions. For an integer $N \geq 1$, define the standard congruence subgroups of $SL_2(\field{Z})$:
\begin{equation*}
    \Gamma_1(N)=\left\{ \begin{pmatrix} a & b \\ c & d \end{pmatrix} \in SL_2(\field{Z}) : \begin{pmatrix} a & b \\ c & d \end{pmatrix} \equiv \begin{pmatrix} 1 & * \\ 0 & 1 \end{pmatrix} \ (\text{mod} \ N)  \right\}
\end{equation*}
\begin{equation*}
    \Gamma_0(N)=\left\{ \begin{pmatrix} a & b \\ c & d \end{pmatrix} \in SL_2(\field{Z}) : \begin{pmatrix} a & b \\ c & d \end{pmatrix} \equiv \begin{pmatrix} * & * \\ 0 & * \end{pmatrix} \ (\text{mod} \ N)  \right\}.
\end{equation*}

Let $M_k(\Gamma_0(N), \chi)$ be the space of modular forms of weight $k \geq 2$, level $N$ and Dirichlet character $\chi : (\field{Z}/N\field{Z})^\times \rightarrow \field{C}^\times$. This is the space of holomorphic functions $f:\mathcal{H}\rightarrow\mathbb{C}$ on the upper half plane $\mathcal{H}$ that  satisfy:
\begin{equation*}
    f[\gamma]_k:=(cz + d)^{-k} f\left(\frac{az+b}{cz+d} \right) = \chi(d)f(z)
\end{equation*}
for all $\gamma=\begin{pmatrix} a & b \\ c & d \end{pmatrix} \in \Gamma_0(N) $, and are such that $f[\alpha]_k$ is holomorphic at $i\infty$ for all $\alpha \in SL_2(\field{Z})$ (i.e.\ the Fourier expansion of $f[\alpha]_k$ is of the form $f[\alpha]_k(z) = \sum_{n=0}^\infty a_n(f) q^n$, with $q=e^{2\pi i z}$). Note that the weight must satisfy $\chi(-1)=(-1)^k$, since $\begin{pmatrix} -1 & 0 \\ 0 & -1 \end{pmatrix}\in\Gamma_0(N)$.

The subspace $S_k(\Gamma_0(N), \chi)$ of cusp forms consists of the forms such that $f[\alpha]_k$ has Fourier coefficient $a_0(f)=0$ for each $\alpha\in SL_2(\field{Z})$. The orthogonal complement of $S_k(\Gamma_0(N),\chi)$ with respect to the Petersson inner product is the Eisenstein subspace $\mathcal{E}_k(\Gamma_0(N), \chi)$. If $k>2$ (the case of interest to us) then a natural basis of this space consists of the normalised Eisenstein series $E_k^{\psi, \phi}(tz)$ for all ordered pairs of Dirichlet characters $\phi,\psi$ of conductors $u,v$ satisfying $\psi\phi = \chi$ and $tuv\mid N$. The Fourier expansion of $ E_k^{\psi,\phi}$ is:
\begin{equation*}
    E_k^{\psi,\phi}(z)=\delta(\psi)\frac{L(1-k, \psi^{-1}\phi)}{2} + \sum_{n=1}^\infty \sigma_{k-1}^{\psi, \phi}(n) q^n,.
\end{equation*}
where $\delta(\psi)=\delta_{\psi,\mathds{1}_1}$ (the trivial character modulo $1$), $L(1-k,\psi^{-1}\phi)\in\mathbb{Q}(\psi,\phi)$ is a special value of the Dirichlet L-function attached to $\psi^{-1}\phi$ and
\begin{equation*}
   \sigma_{k-1}^{\psi,\phi}(n):=\sum_{d\mid n, d>0} \psi(n/d) \phi(d) d^{k-1}
\end{equation*}
is a generalised power divisor sum. The Eisenstein series with $uv = N$ are referred to as being new at level $N$. 

For any given level $N$, we can also decompose $S_k(\Gamma_1(N))$ into new and old subspaces. For any $d\mid N$ we have the map:
\begin{align*}
    i_d: S_k(\Gamma_1(N/d))^2 &\rightarrow S_k(\Gamma_1(N)) \\
    (f,g) &\mapsto f + g[\alpha_d]_k
\end{align*}

\noindent where $\alpha_d = \left(\begin{smallmatrix} d & 0 \\ 0 & 1\end{smallmatrix}\right)$. Then the old subspace is:
\begin{equation*}
 S_k^{\text{old}}(\Gamma_1(N)) = \sum_{p \in \mathcal{P}_N} i_p(S_k(\Gamma_1(N/p))^2),
\end{equation*} where $\mathcal{P}_N=\{p \text{ prime}: p \mid N \}$. The new subspace $S_k^{\text{new}}(\Gamma_1(N))$ is then the orthogonal complement of $S_k^{\text{old}}(\Gamma_1(N))$ with respect to the Petersson inner product, so that 
\begin{equation*}
   S_k(\Gamma_1(N)) = S_k^{\text{old}}(\Gamma_1(N))\oplus S_k^{\text{new}}(\Gamma_1(N)). 
\end{equation*} This induces a decomposition of the space $S_k(\Gamma_0(N),\chi)$ into new and old spaces (lifting from spaces $S_k(\Gamma_0(N/p),\chi')$ with $p\in \mathcal{P}_N$ such that $p\text{ cond}(\chi)\mid N$, and taking $\chi'$ to be the reduction of $\chi$ mod $N/p$).

The space $M_k(\Gamma_0(N),\chi)$ comes equipped with the action of a Hecke algebra, a commutative algebra generated by operators $T_p$ indexed by primes $p$. The action of $T_p$ on the level of Fourier coefficients is as follows: 
\begin{equation*}
    a_n(T_p(f)) = a_{np}(f) + \chi(p)p^{k-1}a_{n/p}(f),
\end{equation*}
where we take $a_{n/p}(f)=0$ if $n/p \notin \field{Z}$. 

The space $S_k(\Gamma_0(N),\chi)$ has a basis of eigenforms for the operators $T_p$ for all $p\nmid N$. We can always normalise an eigenform $f$ so that $a_1(f) = 1$, and in this case $T_p(f) = a_p(f)f$ for each prime $p\nmid N$. For the subspace $S_k^{\text{new}}(\Gamma_0(N),\chi)$ we can find a basis of newforms, eigenforms for the operators $T_p$ for all $p$.

The action of the Hecke algebra on the Eisenstein subspace $\mathcal{E}_k(\Gamma_0(N),\chi)$ is also well understood (see \cite[Proposition 5.2.3]{diamond_shurman_2016}). In particular, the normalised Eisenstein series $E_k^{\psi,\phi}$ are eigenforms for all Hecke operators $T_p$ with $p\nmid N$, with eigenvalues given by:

\begin{equation*}
    T_p(E_k^{\psi,\phi}) = \sigma_{k-1}^{\psi,\phi}(p)E_k^{\psi,\phi} = (\psi(p)+\phi(p)p^{k-1})E_k^{\psi,\phi}.
\end{equation*} If $E_k^{\psi,\phi}$ is new at level $N$ then it is an eigenform for the full Hecke algebra and so the above holds for all $p$.

The field of definition $K_f = \mathbb{Q}(\{a_n(f)\})$ of an eigenform $f\in S_k(\Gamma_0(N),\chi)$ is known to be a number field, and so has a well defined ring of integers $\mathcal{O}_f$. We will often denote by $K_f[\psi,\phi]$ the finite extension generated by the values of $\psi,\phi$ (i.e. roots of unity), and denote by $\mathcal{O}_f[\psi,\phi]$ be the corresponding ring of integers.

By a theorem of Deligne \cite{deligne_1969}, for each prime $\Lambda$ of $\mathcal{O}_f$ there exists a continuous, irreducible $\Lambda$-adic Galois representation \[\rho_{f,\Lambda}: G_{\mathbb{Q}} \rightarrow \text{GL}_2(K_{f,\Lambda})\] which is unramified for $q \nmid NMl$ and satisfies
\begin{align*}
    \text{Tr}(\rho_{f,\Lambda}(\text{Frob}_q)) &=a_q(f) \\ \text{det}(\rho_{f,\Lambda}(\text{Frob}_q)) &=\chi(q)q^{k-1}.
\end{align*} for such primes (here $\text{Frob}_q$ is an arithmetic Frobenius element at $q$). The above properties uniquely determine $\rho_{f,\Lambda}$.
By standard arguments, it is possible to conjugate $\rho_{f,\Lambda}$ so that it takes values in $\text{GL}_2(\mathcal{O}_{f,\Lambda})$ and reduce modulo $\Lambda$ to get a continuous representation \[\overline{\rho}_{f,\Lambda}:G_{\mathbb{Q}} \rightarrow \text{GL}_2(\mathbb{F}_{\Lambda}).\] In general, the reduction depends on the choice of invariant $\mathcal{O}_{f,\Lambda}$-lattice, but the irreducible composition factors are independent of this choice.

\section{A General Conjecture}
\label{section:conj}
From now on, $N, M\geq 1$ are fixed coprime squarefree integers and $\chi$ is a fixed character of conductor $N$. Suppose further that $\phi, \psi$ are characters of conductors $u, v\geq 1$ respectively satisfying $uv=N$, and set $\psi \phi = \chi$. Then $E_k^{\psi,\phi}\in \mathcal{E}_k(\Gamma_0(N),\chi)$ is new at level $N$. We define $\tilde{\chi}$ to be the mod $NM$ Dirichlet character such that $\tilde{\chi}(x) = \chi(x \bmod N)$ for all $x\in(\mathbb{Z}/NM\mathbb{Z})^{\times}$.

The following is a restatement of Conjecture $1.1$. It is a general conjecture concerning Eisenstein congruences between newforms in $S_k^{\text{new}}(\Gamma_0(NM),\tilde{\chi})$ and $E_k^{\psi,\phi}$, providing a wide generalisation of Ramanujan's congruence (as defined earlier, $\mathcal{P}_d$ is the set of prime divisors of $d\geq 1$).

\begin{conj}
\label{conj:general}
Let $k > 2$ and $l>k+1$ be a prime satisfying $l \nmid \varphi(N)NM$. There exists a newform $f \in S_k^{\text{new}}(\Gamma_0(NM), \tilde{\chi})$ and a prime $\lambda\mid l$ of $\mathcal{O}_f[\phi,\psi]$ such that 
\begin{equation*}
    a_q(f) \equiv \psi(q)+\phi(q)q^{k-1} \ \text{mod } \lambda
\end{equation*}
for all primes $q\nmid NM$ if and only if both of the following conditions hold for some $\lambda'|l$ in $\mathbb{Z}[\psi,\phi]$ (satisfying $\lambda|\lambda'$):
\begin{enumerate}
\item $\text{ord}_{\lambda'}(L(1-k, \psi^{-1}\phi) \prod_{p \in \mathcal{P}_M} (\psi(p) - \phi(p)p^k)) > 0$.
\item $\text{ord}_{\lambda'}((\psi(p)-\phi(p)p^k)(\psi(p)-\phi(p)p^{k-2}))>0$ for each $p \in \mathcal{P}_M$.
\end{enumerate}
\end{conj} 

Condition (1) of the conjecture is enough to guarantee the existence of an eigenform in $S_k(\Gamma_0(NM),\tilde{\chi})$ satisfying the congruence (see Theorem \ref{thm:eigenform} below). This condition can be thought of as the analogue of $\text{ord}_{691}\left(-\frac{B_{12}}{24}\right)>0$ in Ramanujan's congruence, but now allowing prime divisors of Euler factors as well as Dirichlet $L$-values (an artifact of our Eisenstein series being lifted from level $N$ to level $NM$). Such congruences coming from divisibility of the Euler factor (as opposed to the complete $L$-value) are often said to be of ``local origin".

Condition (2) can be thought of as measuring how ``new'' the modular form $f$ in the conjecture is, i.e. if the prime $l$ only satisfied condition (2) for all $p \in \mathcal{P}_d$, with $d \mid M$, then we would only expect to find a $d$-newform $f$ satisfying the congruence condition. 

We will prove the reverse implication of this Conjecture, and partial results concerning the direct implication. The major hurdle is that the Conditions (1) and (2) allow us to prove the existence of a $p$-newform $f$ satisfying the congruence condition for each $p \in \mathcal{P}_M$, but we are currently unable to show that these $p$-newforms can in fact be taken to the same genuine newform.

Later, we will discuss the relationship of this conjecture with previous results in this area, for example \cite{dummigan_2007}, \cite{billerey_menares_2016}, \cite{dummigan_fretwell_2014}, as well as connections with the Bloch-Kato Conjecture. We will also see computational examples using MAGMA \cite{magma}.

\subsection{Initial results}
To get us started, we use $E_k^{\psi,\phi}$ to construct various Eisenstein series of level $MN$ having specified constant terms. The following construction and results generalise those found in \cite[Section 1.2.2]{billerey_menares_2016}

For each $m\geq 1$ we define the operator $\alpha_m$, acting on complex valued functions on the upper half plane, by $(\alpha_m f)(z)=f(mz)$. We then consider the collection of Eisenstein series given by
\begin{equation}
\label{eq:E}
    E_{\underline{\delta}}(z)=\left[\prod_{p\in \mathcal{P}_M}(T_p - \delta_p)\right]\alpha_M E_k^{\psi,\phi}\in \mathcal{E}_k(\Gamma_0(NM),\tilde{\chi}),
\end{equation}
where $\underline{\delta} = \{\delta_p\}_{p\mid M}$ and $\delta_p$ is determined by fixing an ordering:
\begin{equation}
\label{eq:delta/eps}
    \{\delta_p, \varepsilon_p\} = \{\psi(p), \phi(p)p^{k-1}\}.
\end{equation} 

\noindent When $M=1$ there is no choice and we define $E_{\underline{\delta}} = E_k^{\psi,\phi}$.

\begin{lemma}
\label{lemma:E}
Each Eisenstein series $E_{\underline{\delta}}$ is a normalised eigenform in \\ $\mathcal{E}_k(\Gamma_0(NM), \tilde{\chi})$. For each prime $p$ we have:
\begin{equation*}
T_p E_{\underline{\delta}}=
\begin{cases}
 \varepsilon_pE_{\underline{\delta}} & \text{if } p \in \mathcal{P}_M \\
 (\psi(p)+\phi(p)p^{k-1})E_{\underline{\delta}} & \text{otherwise} \\
\end{cases}
\end{equation*}
We can also write 
\begin{equation}
\label{eq:altE}
    E_{\underline{\delta}}=\sum_{m \mid M}(-1)^{|\mathcal{P}_m|} \delta_m \alpha_m E_k^{\psi,\phi}.
\end{equation}
where $\delta_m = \prod_{p \in \mathcal{P}_m} \delta_p$ for each $m\mid M$.
\end{lemma}
\begin{proof}
If $p \notin \mathcal{P}_M$ then

\begin{align*}
    a_n(T_p(\alpha_M E_k^{\psi, \phi})) &= a_{np}(\alpha_M E_k^{\psi, \phi}) + \tilde{\chi}(p)p^{k-1}a_{n/p}(\alpha_M E_k^{\psi,\phi}) \\
    &= a_{np/M}(E_k^{\psi,\phi}) + \chi(p)p^{k-1}a_{n/pM}(E_k^{\psi,\phi}) \\
    &= a_{n/M}(T_pE_k^{\psi,\phi}) \\
    &= a_n(\alpha_M (T_p E_k^{\psi,\phi})).
\end{align*}

\noindent Hence, we see that $T_p \alpha_M E_k^{\psi,\phi} = \alpha_M T_p E_k^{\psi,\phi}$. It follows that 
\begin{align*}
    T_pE_{\underline{\delta}}&=T_p \left[\prod_{q\in \mathcal{P}_M} (T_q - \delta_q)\right]\alpha_ME_k^{\psi,\phi} \\
    &=\left[\prod_{q\in \mathcal{P}_M} (T_q - \delta_q)\right]\alpha_M T_pE_k^{\psi,\phi} \\
    &=(\psi(p)+\phi(p)p^{k-1})E_{\underline{\delta}}
\end{align*}

\noindent If $p \in \mathcal{P}_M$ we find that: \begin{align*}T_pE_{\underline{\delta}}&=T_p \left[\prod_{q\in \mathcal{P}_M} (T_q - \delta_q)\right]\alpha_ME_k^{\psi,\phi} \\
    &=\left[\prod_{q\in \mathcal{P}_{M/p}} (T_q - \delta_q)\right](T_p^2-\delta_pT_p)\alpha_M E_k^{\psi,\phi}.\end{align*} and we must figure out the action of the final operator on $E_k^{\psi,\phi}$. We do this by first proving the claim that for each $m\mid M$:
\begin{equation*}
T_p\alpha_m E_k^{\psi, \phi} = 
    \begin{cases}
    \alpha_{m/p}E_k^{\psi,\phi} & \text{if } p \in 
    \mathcal{P}_m \\
    (\psi(p)+\phi(p)p^{k-1})\alpha_{m}E_k^{\psi,\phi}-\chi(p)p^{k-1}\alpha_{mp}E_k^{\psi,\phi} & \text{if } p \notin 
    \mathcal{P}_m
    \end{cases}
\end{equation*}
To prove the claim, note that
\begin{align*}
    a_n(T_p\alpha_m E_k^{\psi,\phi})&=a_{np}(\alpha_m E_k^{\psi,\phi}) + \tilde{\chi}(p)p^{k-1}a_{n/p}(\alpha_m E_k^{\psi,\phi}) \\
    &= a_{np}(\alpha_m E_k^{\psi,\phi}) \\
    &= \sigma_{k-1}^{\psi,\phi}\left(\frac{np}{m}\right).
\end{align*}
(here, $\tilde{\chi}(p)$ vanishes since $\tilde{\chi}$ has modulus $NM$ and $p \in \mathcal{P}_M$).

\noindent When $p \in \mathcal{P}_m$ we have $a_n(T_p\alpha_m E_k^{\psi,\phi})=\sigma_{k-1}^{\psi,\phi}\left(\frac{np}{m}\right)=a_n(\alpha_{m/p}E_k^{\psi,\phi})$, and when $p\notin \mathcal{P}_m$ we use the fact that:
\begin{equation}
\label{eq:powerdiv}
    \sigma_{k-1}^{\psi,\phi}(np)+\chi(p)p^{k-1}\sigma_{k-1}^{\psi,\phi}(n/p) = (\psi(p)+\phi(p)p^{k-1})\sigma_{k-1}^{\psi,\phi}(n)
\end{equation}
to get
\begin{equation*}
    \sigma_{k-1}^{\psi,\phi}\left(\frac{np}{m}\right) = (\psi(p)+\phi(p)p^{k-1})\sigma_{k-1}^{\psi,\phi}\left(\frac{n}{m}\right)-\chi(p)p^{k-1}\sigma_{k-1}^{\psi,\phi}\left(\frac{n}{mp}\right),
\end{equation*}
so that 
\begin{equation*}
   a_n(T_p\alpha_m E_k^{\psi,\phi})=(\psi(p)+\phi(p)p^{k-1})a_n(\alpha_mE_k^{\psi,\phi})-\chi(p)p^{k-1}a_n(\alpha_{mp}E_k^{\psi,\phi}).
\end{equation*}

\noindent The claim follows and so
\begin{align*}
    (T_p^2-\delta_pT_p)\alpha_ME_k^{\psi,\phi} &= T_p\alpha_{M/p}E_k^{\psi,\phi} - \delta_p\alpha_{M/p}E_k^{\psi,\phi} \\
    &=(\psi(p)+\phi(p)p^{k-1})\alpha_{M/p}E_k^{\psi,\phi}-\chi(p)p^{k-1}\alpha_ME_k^{\psi,\phi} - \delta_p\alpha_{M/p}E_k^{\psi,\phi} \\
    &=\varepsilon_p\alpha_{M/p}E_k^{\psi,\phi} - \chi(p)p^{k-1}\alpha_ME_k^{\psi,\phi} \\
    &=\varepsilon_p(T_p-\delta_p)\alpha_M E_k^{\psi,\phi}.
\end{align*}
Therefore, when $p\in \mathcal{P}_M$ we have that $T_pE_{\underline{\delta}}=\varepsilon_pE_{\underline{\delta}}$, as required. 

Equation \eqref{eq:altE} holds by the Inclusion-Exclusion Principle and the claim. The fact that $E_{\underline{\delta}}$ is normalised now follows, since only the term with $m=1$ contributes to the $q$ coefficient of $E_{\underline{\delta}}$, and $E_k^{\psi,\phi}$ is normalised.
\end{proof}

We now state a Proposition which allows us to determine the value of $E_{\underline{\delta}}$ at all cusps. The proof can be found in \cite[Proposition 3.1.2]{spencer_2018} or \cite[Proposition 4]{billerey_menares_2017}.

\begin{propn}[Spencer]
\label{prop:spencer}
If $m \geq 1$ is coprime to $N$ and $\gamma = \left(\begin{smallmatrix} a & \beta \\ b & \delta \end{smallmatrix}\right) \in SL_2(\field{Z})$ then the constant term of $(\alpha_mE_k^{\psi,\phi})[\gamma]_k$ is given by
\begin{equation*}
    a_0((\alpha_m E_k^{\psi,\phi})[\gamma]_k) = \begin{cases}
    -\frac{g(\psi\phi^{-1})}{g(\phi^{-1})}\frac{\phi^{-1}(m^\prime a) \psi\left(\frac{-b^\prime}{v}\right)}{u^k m^{\prime k}}\frac{L(1-k, \psi^{-1}\phi)}{2} & \text{ if }v\mid b'\\ 0 & \text{otherwise}
    \end{cases}
\end{equation*}
where $b^\prime = \frac{b}{gcd(b,m)}$, $m^\prime=\frac{m}{gcd(b,m)}$ and $g(\phi)=\sum_{n=0}^{v-1}\phi(n)e^{\frac{2\pi i n}{v}}$ is the Gauss sum of $\phi$. 
\end{propn}

For convenience we will write:
\begin{equation*}
    C_{\gamma}=-\frac{g(\psi\phi^{-1})}{g(\phi^{-1})}\frac{\phi^{-1}(a) \psi\left(\frac{-b}{v}\right)}{u^k }\frac{L(1-k, \psi^{-1}\phi)}{2}
\end{equation*}
Using this and the notation in Proposition \ref{prop:spencer}, we can write a formula for the constant term of $E_{\underline{\delta}}[\gamma]_k$.
\begin{corollary}
\label{cor:cusp}
If $\gamma = \left(\begin{smallmatrix} a & \beta \\ b & \delta \end{smallmatrix}\right) \in SL_2(\field{Z})$ and $M' = \frac{M}{\text{gcd}(M,b)}$ then the constant term of $E_{\underline{\delta}}[\gamma]_k$ is
\begin{equation*}
    a_0(E_{\underline{\delta}}[\gamma]_k) = \begin{cases}C_{\gamma} \prod_{p\in \mathcal{P}_{M'}}(1 - \delta_p \phi^{-1}(p)p^{-k})\prod_{p \in \mathcal{P}_{M/M'}} (1 - \delta_p \psi^{-1}(p))  & \text{ if } v\mid b'\\ 0& \text{ otherwise}\end{cases}.
\end{equation*}
\end{corollary}
\begin{proof}
By Lemma \ref{lemma:E}, we can write $E$ as
\begin{equation*}
    E_{\underline{\delta}}=\sum_{m \mid M}(-1)^{|\mathcal{P}_m|} \delta_m \alpha_m E_k^{\psi,\phi}.
\end{equation*}
It is clear from Proposition \ref{prop:spencer} that the constant term of $E_{\underline{\delta}}$ is zero if $v\nmid b'$ and so it suffices to consider the case $v \mid b^\prime$. First we use Proposition \ref{prop:spencer} to evaluate the constant term of $(-1)^{|\mathcal{P}_m|}\delta_m \alpha_m E_k^{\psi, \phi}[\gamma]_k$ for a fixed $m \mid M$ as follows:
\begin{align*}
    &(-1)^{|\mathcal{P}_m| +1}\frac{g(\psi \phi^{-1})}{g(\phi^{-1})} \frac{\phi^{-1}\left(m' a \right)\psi \left(\frac{-b'}{v}\right)}{u^k m'^k} \frac{L(1-k, \psi^{-1}\phi)}{2} \delta_m \\
    &= C_{\gamma} (-1)^{|\mathcal{P}_m|} \frac{\phi^{-1}(m')\psi^{-1}\left(\frac{m}{m'}\right)}{m'^k} \delta_m\\ &=  C_{\gamma} (-1)^{|\mathcal{P}_m|} \left(\prod_{p\in \mathcal{P}_{m'}} \delta_p\phi^{-1}(p)p^{-k} \right) \left( \prod_{p \in \mathcal{P}_{m/m'}} \delta_p\psi^{-1}(p) \right).\label{eq:constant}
\end{align*}
It follows that the constant term of $E_{\underline{\delta}}[\gamma]_k$ is:
\begin{align*}
    a_0(E_{\underline{\delta}}[\gamma]_k) &=C_{\gamma} \sum_{m \mid M} \left[ (-1)^{|\mathcal{P}_m|} \left(\prod_{p\in \mathcal{P}_{m'}} \delta_p\phi^{-1}(p)p^{-k}\right) \left( \prod_{p \in \mathcal{P}_{m/m'}} \delta_p\psi^{-1}(p) \right)\right] \\
    &= C_{\gamma} \prod_{p \in \mathcal{P}_{M'}} (1-\delta_p\phi^{-1}(p)p^{-k})  \prod_{p \in \mathcal{P}_{M/M'}} (1-\delta_p\psi^{-1}(p))
\end{align*}
by the Inclusion-Exclusion Principle.
\end{proof}

\subsection{Necessary conditions for a newform Eisenstein congruence.}
We have now constructed various Eisenstein series $E_{\underline{\delta}}$ whose constant coefficients contain terms which look somewhat like those in Condition (1) of Conjecture \ref{conj:general}. We are now in a position to use this to prove the reverse implication of the conjecture.

\begin{theorem}
\label{thm:reverse}
Suppose $f \in S_k^{\text{new}}(\Gamma_0(NM), \tilde{\chi})$ is a newform satisfying the congruence \[a_q(f) \equiv \psi(q) + \phi(q)q^{k-1} \ \text{mod } \lambda\] for all primes $q \nmid NM$ and for some prime $\lambda\mid l$ of $\mathcal{O}_f[\psi,\phi]$ (with the assumptions $k>2$, $l>k+1$ and $l\nmid NM$). Then conditions (1) and (2) of Conjecture \ref{conj:general} hold.
\end{theorem}

\begin{proof}
Consider the $\lambda$-adic Galois representation attached to $f$ given by \[\rho_{f,\lambda}: G_{\mathbb{Q}}\rightarrow \text{GL}_2(K_{f,\lambda_0}) \hookrightarrow \text{GL}_2(K_f[\psi,\phi]_{\lambda}),\] where the first arrow is the usual $\lambda_0$-adic Galois representation attached to $f$ (with $\lambda_0$ the unique prime of $K_f$ lying below $\lambda$). This may be conjugated to take values in $\text{GL}_2(\mathcal{O}_f[\psi,\phi]_\lambda)$. The congruence then implies that $\rho_{f,\lambda}$ is residually reducible mod $\lambda$ (by the Cebotarev density theorem and the Brauer-Nesbitt theorem, noting that $l = \text{char}(\mathbb{F}_{\lambda}) > 2$):
\begin{equation*}
\overline{\rho}_{f,\lambda}  \sim   \begin{pmatrix} \overline{\psi} & *\\ 0 & \overline{\phi} \chi_l^{k - 1} \end{pmatrix}
\end{equation*}
so that the semisimplification is given by
\begin{equation}
\label{eq:rep}
    \overline{\rho}_{f,\lambda}^{ss} \sim \overline{\psi} \oplus \overline{\phi} \chi_l^{k-1}.
\end{equation}
(i.e.\ $\overline{\rho}_{f,\lambda}$ has composition factors $\{\overline{\psi}, \overline{\phi}\chi_l^{k-1}\}$). Here $\chi_{l}:G_{\mathbb{Q}}\rightarrow \mathbb{F}_{l}^{\times}$ is the mod $l$ cyclotomic character.

For each $p\in \mathcal{P}_M$ it is known that that the composition factors of $\overline{\rho}_{f,\lambda}$ locally at $p$ are given by: 
\begin{equation*}
\overline{\rho}^{\text{ss}}_{f,\lambda}\mid_{W_{\mathbb{Q}_p}} \sim   \left(\eta\chi_l\oplus\eta\right)\mid_{W_{\mathbb{Q}_p}}
\end{equation*}
 Here $W_{\mathbb{Q}_p}$ is the local Weil group at $p$ and $\eta : W_{\mathbb{Q}_p}\rightarrow\mathbb{F}_\lambda^{\times}$ is the unramified character such that $\eta(\text{Frob}_p) \equiv a_p(f) \text{ mod } \lambda$ (e.g.\ see \cite[Theorem 4.2.7, (3)(b)]{hida}). By comparing determinants we find that $\eta^2\chi_l = \chi_l^{k-1}$, so that $\eta = \mu\chi_l^{(k-2)/2}$ for some unramified quadratic character $\mu$. This leads to the following equivalence for each $p\in \mathcal{P}_M$:
\begin{equation*}
   \left(\mu\chi_l^{k/2}\oplus\mu\chi_l^{(k-2)/2}\right)\mid_{W_{\mathbb{Q}_p}} \sim \left(\overline{\psi}\oplus \,\overline{\phi} \chi_l^{k-1}\right)\mid_{W_{\mathbb{Q}_p}}.
\end{equation*}
 There are only two possibilities (note also that $l\neq p$ since $p\in \mathcal{P}_M$):

\vspace{0.1in}
\begin{enumerate}
    \item[(A)] $\overline{\psi}\mid_{W_{\mathbb{Q}_p}} = \mu \chi_l^{k/2}\mid_{W_{\mathbb{Q}_p}}$\qquad\text{and}\qquad  $\overline{\phi} \chi_l^{k-1}\mid_{W_{\mathbb{Q}_p}} = \mu  \chi_l^{(k-2)/2}\mid_{W_{\mathbb{Q}_p}}$.\\
    
    \noindent Evaluating at $\text{Frob}_p$ gives 
    \begin{align*}\psi(p) &\equiv \mu(p)p^{k/2} \ \text{mod } \lambda\\ \phi(p)p^{k-1} &\equiv \mu(p)p^{(k-2)/2} \ \text{mod } \lambda,\end{align*} so that 
    \begin{equation*}
        \psi(p) - \phi(p)p^k \equiv 0 \ \text{mod } \lambda.
    \end{equation*}
    \item[(B)] $\overline{\psi}\mid_{W_{\mathbb{Q}_p}} = \mu \chi_l^{(k-2)/2}\mid_{W_{\mathbb{Q}_p}}$\qquad\text{and}\qquad $\overline{\phi} \chi_l^{k-1}\mid_{W_{\mathbb{Q}_p}} = \mu \chi_l^{k/2}\mid_{W_{\mathbb{Q}_p}}$. \\

    \noindent Evaluating at $\text{Frob}_p$ gives \begin{align*}\psi(p) &\equiv \mu(p)p^{(k-2)/2} \ \text{mod } \lambda\\ \phi(p)p^{k-1} &\equiv \mu(p)p^{k/2} \ \text{mod } \lambda,\end{align*} so that we have both of the following
    \begin{align*}
        \psi(p) - \phi(p)p^{k-2} &\equiv 0 \ \text{mod } \lambda\\ \psi(p) &\equiv a_p(f) \bmod \lambda.
    \end{align*}
\end{enumerate}

\noindent To summarise, one of the following must hold for each $p\in \mathcal{P}_M$:
\begin{enumerate}
    \item[(A)] $ \psi(p) - \phi(p)p^k \equiv 0 \ \text{mod } \lambda.$
    \item[(B)] $\psi(p) - \phi(p)p^{k-2} \equiv 0 \ \text{mod } \lambda$ and $a_p(f) \equiv \psi(p) \ \text{mod } \lambda$.
\end{enumerate}
Taking norms down to $\mathbb{Z}[\psi,\phi]$ gives Condition (2) (i.e. divisibility by $\lambda'$). It remains to prove that Condition (1) holds: \[\text{ord}_{\lambda'}\left(L(1-k, \psi^{-1}\phi) \prod_{p \in \mathcal{P}_M} (\psi(p) - \phi(p)p^k)\right) > 0.\] First note that this is immediate if there exists a prime $p \in \mathcal{P}_M$ satisfying case (A), since $l>k+1$ and $l\nmid N$. We assume from now on that case (B) is satisfied for each $p \in \mathcal{P}_M$.

Consider the Eisenstein series $E_{\underline{\delta}}$ corresponding to the choice $\delta_p=\phi(p)p^{k-1}$ for each $p \in \mathcal{P}_M$. We claim that for all primes $p\neq l$, the following congruence holds
\begin{equation*}
    a_p(E_{\underline{\delta}}) \equiv a_p(f) \ \text{mod } \lambda.
\end{equation*}
This is true for each prime $p\in \mathcal{P}_M$, since by Lemma \ref{lemma:E} we have \[a_p(E_{\underline{\delta}}) = \varepsilon_p = \psi(p) \equiv a_p(f) \bmod \lambda .\]  For each prime $p\in \mathcal{P}_N$ the form $f$ is $p$-new, and another comparison of local composition factors gives (e.g.\ see \cite[Theorem 4.2.7, (3)(a)]{hida})):
\begin{equation*}
 (\mu_1\oplus\mu_2)\mid_{W_{\mathbb{Q}_p}} \sim \left(\overline{\psi} \oplus \overline{\phi} \chi_l^{k-1}\right)\mid_{W_{\mathbb{Q}_p}}.
\end{equation*}
with $\mu_1,\mu_2: W_{\mathbb{Q}_p}\rightarrow\mathbb{F}_\lambda^{\times}$ characters of conductors $1$ and $p$ satisfying \[a_p(f) \equiv \mu_1(p)+\mu_2(p) \ \text{mod } \lambda.\] It follows that, for such primes \[a_p(f) \equiv \psi(p) + \phi(p)p^{k-1} \equiv a_p(E_{\underline{\delta}}) \ \text{mod } \lambda.\] For all other primes the claim follows from Lemma \ref{lemma:E} and the assumption that $f$ satisfies the congruence.
 
By the claim, and the fact that $f$ and $E_{\underline{\delta}}$ are both normalised Hecke eigenforms, we get the following congruence for all $n$ coprime to $l$
\begin{equation*}
    a_n(E_{\underline{\delta}}) \equiv a_n(f) \ \text{mod } \lambda
\end{equation*}
Applying the theta operator $\Theta = q\frac{d}{dq}$ in \cite{serre_2003}, we obtain $\Theta(E_{\underline{\delta}}) \equiv \Theta(f) \ \text{mod } \lambda$. However, the theta operator is injective for $l > k+1$ \cite[Corollary 3]{katz_1977}, and so we have that $E_{\underline{\delta}} \equiv f \ \text{mod } \lambda$. Since $f$ is a cusp form it must be that $E_{\underline{\delta}}$ must vanish at all cusps mod $\lambda$. Choosing any $\gamma \in SL_2(\field{Z})$ with lower left entry $b\neq 0$ such that $\mathcal{P}_M\subseteq \mathcal{P}_b$, we then find that 
\begin{equation*}
    \text{ord}_\lambda(a_0(E_{\underline{\delta}}[\gamma]_k)) = \text{ord}_\lambda\left(C_{\gamma}\left(\prod_{p \in \mathcal{P}_M} (1-\psi^{-1}(p)\phi(p)p^{k-1})\right)\right)>0.
\end{equation*} by Corollary \ref{cor:cusp} (note that choosing $b=0$ gives $C_{\gamma} = 0$, so that the divisibility condition is automatic).
Using the facts that $\frac{g(\psi\phi^{-1})}{g(\phi)}$, $\phi^{-1}(a), \psi^{-1}(M)$ and $\psi\left(\frac{-b}{v}\right)$ are units in $\field{Z}[\psi, \phi]$, and that $l \nmid 2u$ by assumption, we find that:
\begin{equation*}
    \text{ord}_{\lambda}\left(L(1-k, \psi^{-1}\phi)\prod_{p \in \mathcal{P}_M} (\psi(p)-\phi(p)p^{k-1})\right)>0.
\end{equation*}

\noindent By assumption, we also have that $\psi(p) - \phi(p)p^{k-2} \equiv 0 \bmod \lambda$ for each $p\in\mathcal{P}_M$, and so we have the implication: \[\text{ord}_{\lambda}(\psi(p)-\phi(p)p^{k-1}) > 0 \iff \text{ord}_{\lambda}(\psi(p)-\phi(p)p^k) > 0,\] since in both cases we have that $p\equiv 1 \bmod \lambda$. The following divisibility condition is then apparent:  \begin{equation*}
    \text{ord}_{\lambda}\left(L(1-k, \psi^{-1}\phi)\prod_{p \in \mathcal{P}_M} (\psi(p)-\phi(p)p^k)\right)>0.
\end{equation*}
Condition (1) follows again by taking the norm down to $\mathbb{Z}[\psi,\phi]$.
\end{proof}

\subsection{Sufficient conditions for a newform Eisenstein congruence.}

We now give partial results towards the direct implication of Conjecture \ref{conj:general}. First, we prove that Condition (1) guarantees the existence of an eigenform in $S_k(\Gamma_0(NM),\tilde{\chi})$ satisfying the congruence. The following is the necessary extension of \cite[Theorem 2.10]{dummigan_spencer} and \cite[Theorem 3.0.1]{spencer_2018}.

\begin{theorem}
\label{thm:eigenform}
Let $k > 2$ and $\lambda' \nmid 6NM$ be a prime of $\mathbb{Z}[\psi,\phi]$ such that \[\text{ord}_{\lambda'}\left(L(1-k, \psi^{-1} \phi) \prod_{p \in \mathcal{P}_M}(\psi(p) - \phi(p)p^k)\right)>0.\] There exists a normalised Hecke eigenform $f \in S_k(\Gamma_0(NM), \tilde{\chi})$ and a prime $\lambda\mid \lambda'$ of $\mathcal{O}_f[\psi,\phi]$ such that for all primes $q \nmid NM$, 
\begin{equation*}
    a_q(f) \equiv \psi(q)+\phi(q)q^{k-1} \ \text{mod } \lambda.
\end{equation*}
\end{theorem}

\begin{proof}
We follow a standard argument (see for example \cite[Theorem 3.0.1]{spencer_2018}). Consider the Eisenstein series $E_{\underline{\delta}}$ corresponding to the choice $\delta_p=\psi(p)$ for all $p \in \mathcal{P}_M$. Let $R = \mathbb{Z}[\frac{1}{NM},\zeta_{NM}, \psi, \phi]$, where $\zeta_{NM}$ is a primitive $NM$-th root of unity. Since the $q$-expansions of $E_{\underline{\delta}}$ at all cusps have coefficients lying in $R$, the $q$-expansion principle \cite[Corollary 1.6.2]{katz_1973} implies that $E_{\underline{\delta}} \in M_k(\Gamma_1(NM), R)$ (the space of Katz modular forms of weight $k$ and level $\Gamma_1(NM)$ over $R$. See Chapter 1 of \cite{katz_1973} for more details).  Furthermore, Corollary \ref{cor:cusp} tells us that for each $\gamma\in\text{SL}_2(\mathbb{Z})$, the constant term of $E_{\underline{\delta}}[\gamma]_k$ is either $0$ or
\begin{equation*}
   a_0(E_{\underline{\delta}}[\gamma]_k) = C_{\gamma} \prod_{p \in \mathcal{P}_M} (1 - \psi(p) \phi^{-1}(p)p^{-k}) = (-1)^{|\mathcal{P}_M|}\frac{C_{\gamma}}{M^k}\prod_{p\in\mathcal{P}_M}(\psi(p)-\phi(p)p^k).
\end{equation*}
In either case $\text{ord}_{\lambda'}(a_0(E_{\underline{\delta}}[\gamma]_k)) > 0$, and so the mod $\lambda'$ reduction of $E_{\underline{\delta}}$ is a well defined cusp form $\overline{E} \in S_k(\Gamma_1(NM), \mathbb{F}_{\lambda^\prime})$. This is an eigenform having $T_q$ eigenvalue $\psi(q)+\phi(q)q^{k-1} \bmod \lambda'$ for each $q \nmid NM$. See \cite[Section 1.11]{katz_1973} for the definition of Hecke operator in the Katz setting and compatibility with classical Hecke operators. Note that since $\mathbb{F}_{\lambda^\prime} \subset \overline{\mathbb{F}}_l$ we can view $\overline{E} \in S_k(\Gamma_1(NM), \overline{\mathbb{F}}_l)$.

By \cite[Lemma 1.9]{edixhoven_1997} and the fact that $k > 2$, we have that the natural reduction map $S_k(\Gamma_1(NM), \overline{\mathbb{Z}}_l)\rightarrow S_k(\Gamma_1(NM), \overline{\mathbb{F}}_l)$ is surjective. It follows that $\overline{E}$ must be the reduction of some $g \in S_k(\Gamma_1(NM), \mathcal{O}_{\lambda^{\prime\prime}})$, with $\mathcal{O}_{\lambda^{\prime\prime}}$ the ring of integers of some finite extension $K_{\lambda^{\prime\prime}}$ of $\mathbb{Q}_l$. Note that $g$ is a characteristic zero cusp form, but is not necessarily an eigenform. However, an application of the Deligne-Serre lifting lemma \cite[Lemma 6.11]{deligne_1974} gives the existence of an eigenform $f^\prime \in S_k(\Gamma_1(NM), \mathcal{O}_{\lambda})$, with $\mathcal{O}_\lambda$ the ring of integers of some finite extension $K_\lambda$ of $K_{\lambda^{\prime\prime}}$, whose eigenvalues are congruent to $\psi(q) + \phi(q)q^{k-1} \bmod \lambda$ for all $q \nmid NM$. Finally, Carayol's lemma \cite[Proposition 1.10]{edixhoven_1997} shows that such an $f^\prime$ arises from an eigenform $f \in S_k(\Gamma_0(NM), \tilde{\chi})$ (as opposed to some $f \in S_k(\Gamma_0(NM), \chi^\prime)$ with $\chi^\prime \equiv \tilde{\chi} \bmod \lambda$). This $f$ is the desired eigenform satisfying:
\begin{equation*}
    a_q(f) \equiv \psi(q) + \phi(q)q^{k-1} \bmod \lambda
\end{equation*}
for all $q \nmid NM$.
\end{proof}

Taking $M=1$ in Theorem \ref{thm:eigenform} gives a result of Dummigan \cite[Proposition 2.1]{dummigan_2007}. As remarked in the paper, the eigenform satisfying the congruence must be new (since $\chi$ has conductor $N$). This completes the proof of Conjecture \ref{conj:general} in the case $M=1$. We will now see that the case $M=p$ prime can also be fully proved.

\begin{theorem}
\label{thm:prime}
 If $M=p$ is prime then Conjecture \ref{conj:general} is true.
\end{theorem}

\begin{proof}
Theorem \ref{thm:reverse} provides the reverse implication and so it suffices to prove the direct implication. By Theorem \ref{thm:eigenform}, Condition (1) provides a level $Np$ eigenform $f_0\in S_k(\Gamma_0(Np), \tilde{\chi})$ and a prime $\lambda_0\mid \lambda'$ of $\mathcal{O}_f[\psi,\phi]$ satisfying \[a_q(f_0) \equiv \psi(q)+\phi(q)q^{k-1} \ \text{mod } \lambda_0,\] for all $q\nmid Np$. We may assume that $f_0$ is an oldform, otherwise we are done. Since $\tilde{\chi}$ has conductor $N$, $f_0$ must be a lift of an eigenform $f_1\in S_k^{\text{new}}(\Gamma_0(N),\chi)$. By the Cebotarev density theorem, we have that $\overline{\rho}_{f_1,\lambda_0}\sim\overline{\rho}_{f_0,\lambda_0}$. As earlier, the congruence implies that $\overline{\rho}_{f_0,\lambda_0}^{\text{ss}}\sim\overline{\psi} \oplus \overline{\phi} \chi_l^{k-1}$. Since $l \neq p$ we see that:
\begin{equation*}
    a_p(f_1) \equiv \psi(p) + \phi(p)p^{k-1} \ \text{mod } \lambda_0.
\end{equation*} 

\noindent By Condition (2), we also have that one of the following holds: \begin{align*}\psi(p) &\equiv \phi(p)p^k \ \text{mod } \lambda_0\\ \psi(p) &\equiv \phi(p)p^{k-2} \ \text{mod } \lambda_0,\end{align*} so that
\begin{equation*}
    a_p(f_1) \equiv \psi(p)+\phi(p)p^{k-1} \equiv 
    \begin{cases}
    \psi(p)(1 + p^{-1}) \ \text{mod } \lambda_0 & \text{if } \psi(p) \equiv \phi(p)p^k \ \text{mod } \lambda_0 \\
    \psi(p)(1 + p) \ \text{mod } \lambda_0 & \text{if } \psi(p) \equiv \phi(p)p^{k-2} \ \text{mod } \lambda_0 
    \end{cases}
\end{equation*}

We now claim that the following congruence condition holds:
\begin{equation}
\label{eq:diamond}
    a_p(f_1)^2 \equiv \chi(p)p^{k-2}(1+p)^2 \ \text{mod } \lambda_0.
\end{equation}
Indeed, if $\psi(p) \equiv \phi(p)p^k \ \text{mod } \lambda_0$ then
\begin{align*}
    \chi(p)p^{k-2}(1+p)^2 \equiv \psi(p)\phi(p)p^{k-2}(1+p)^2
    \equiv \psi^2(p)p^{-2}(1+p)^2
    &\equiv \psi^2(p)(1+p^{-1})^2 \\ &\equiv a_p(f_1)^2 \bmod \lambda_0
\end{align*}
Alternatively, if $\psi(p) \equiv \phi(p)p^{k-2} \ \text{mod } \lambda_0$ then
\begin{align*}
 \chi(p)p^{k-2}(1+p)^2 \equiv \psi(p)\phi(p)p^{k-2}(1+p)^2 
    &\equiv \psi^2(p)(1+p)^2 \\ &\equiv a_p(f_1)^2 \bmod  \lambda_0.
\end{align*}
A theorem of Diamond \cite[Theorem 1]{diamond_1991} now implies the existence of a normalised $p$-newform $f \in S_k^{p\text{-new}}(\Gamma_0(Np), \tilde{\chi})$ and a prime $\lambda \mid \lambda_0$ of $\mathcal{O}_{f_0,f}[\psi,\phi]$ satisfying
\begin{equation*}
    a_q(f) \equiv a_q(f_1) \ \text{mod } \lambda
\end{equation*}
for all primes $q\nmid Np$ (the conditions on $l$ imply that $l\mid\frac{1}{2}\varphi(N)Np(k-2)!$). In fact, we must have that $f\in S_k^{\text{new}}(\Gamma_0(Np),\tilde{\chi})$, since $\tilde{\chi}$ has conductor $N$ and $p\nmid N$. This newform satisfies the required congruence by construction.
\end{proof}

The above argument highlights the bottleneck in trying to prove the direct implication in general. Conditions (1) and (2) still imply the level raising condition for each $p\mid M$, but this only allows us to find ``local" newforms satisfying the congruence, i.e.\ a $p$-newform for each $p|M$. There seems to be no clear way to prove the existence of a ``global" newform satisfying the congruence.

\subsection{Comparison with known results}

In the special case of $N=1$ (so that $\chi = \mathds{1}_1$) Conjecture \ref{conj:general} agrees with Conjecture $4.1$ of Dummigan and Fretwell \cite{dummigan_fretwell_2014} and Conjecture $3.2$ of Billerey and Menares \cite{billerey_menares_2016}.

In the special case of arbitrary square-free $N$ and $M=p$ prime, Theorem \ref{thm:eigenform} becomes Theorem $3.0.1$ of Spencer's thesis \cite{spencer_2018}. Newform congruences were not explored in this thesis, and so Theorem \ref{thm:prime} complements this Theorem well.

\subsection{Low weight}

For weight $k=2$, the analogue of Theorem \ref{thm:eigenform} is expected to be true in the case $N> 1$. However, when $N=1$ the condition can fail to provide an eigenform congruence. For example, when $N=1$ and $M = p\geq 5$ is prime then a famous result of Mazur \cite{mazur} says that eigenform congruences only arise when $\text{ord}_{l}\left(\frac{p-1}{12}\right)>0$, as opposed to $\text{ord}_{l}(p^2-1)>0$. Work of Ribet and Yoo considers results for more general levels \cite{ribetyoo}.

Conjecture \ref{conj:general} is also invalid in general for weight $k=2$. Even when an eigenform congruence exists there may not exist a newform satisfying the congruence, despite the fact that it is possible for Condition (2) to be automatically satisfed. It would be interesting to see what the analogue of Conjecture \ref{conj:general} is in this case.

It would be very interesting to see if there are analogues of our results for weight $1$ modular forms. The existence of such eigenform congruences has been studied for $M=p$ prime in \cite{spencer_2018}, but very little seems to be known beyond this.

\section{Relation with the Bloch-Kato Conjecture}
\label{section:BK}
In this section, we relate Conjecture \ref{conj:general} to the Bloch-Kato Conjecture. Throughout we assume Conditions (1) and (2) and fix an eigenform $f\in S_k(\Gamma_0(NM), \tilde{\chi})$ satisfying the congruence mod $\lambda$ (guaranteed to exist by Theorem \ref{thm:eigenform}).

As earlier, the congruence implies that the composition factors of $\overline{\rho}_{f,\lambda}$ are given by: \[\overline{\rho}_{f,\lambda}^{\text{ss}} \sim \overline{\psi}\oplus\overline{\phi}\chi_l^{k-1},\] realised on the one dimensional $\mathbb{F}_\lambda[G_{\mathbb{Q}}]$-modules $\mathbb{F}_\lambda(\psi)$ and $\mathbb{F}_\lambda (1-k)(\phi)$ , a Tate twist of $\mathbb{F}_{\lambda}(\phi)$, i.e.\ $\text{Frob}_p$ acts by multiplication by $\phi(p)p^{1-k} \bmod \lambda$ if $\lambda\nmid p$. 

By a result of Ribet \cite{ribet_1976}, we can choose the invariant $\mathcal{O}_{f,\lambda}$-lattice defining $\rho_{f,\lambda}$ in such a way that $\overline{\rho}_{f,\lambda}$ is realised on an $\mathbb{F}_{\lambda}$-vector space $V$ such that
\begin{equation*}
    0 \longrightarrow \mathbb{F}_\lambda(1-k)(\phi) \overset{\iota}\longrightarrow V \overset{\pi}\longrightarrow \mathbb{F}_\lambda(\psi) \longrightarrow 0
\end{equation*}
is a non-split extension of $\mathbb{F}_\lambda[G_{\mathbb{Q}}]$-modules. Twisting by $\psi^{-1}$ gives a non-split extension $V(\psi^{-1})$ which defines a non-trivial class: \[c\in H^1(\mathbb{Q},\mathbb{F}_{\lambda}(1-k)(\psi^{-1}\phi)) = \text{Ext}_{G_{\mathbb{Q}}}^1(\mathbb{F}_{\lambda},\mathbb{F}_{\lambda}(1-k)(\psi^{-1}\phi)).\]

Consider the $\mathbb{F}_\lambda[G_{\mathbb{Q}}]$-module $A_{f,\lambda}^{\psi,\phi}= (K_f[\psi,\phi]_{\lambda} / \mathcal{O}_f[\psi,\phi]_{\lambda})(1-k)(\psi^{-1}\phi)$ and let $A_f^{\psi,\phi}[\lambda]$ be the kernel of multiplication by $\lambda$ (abusing notation slightly, we let $\lambda$ be a uniformiser). It follows that \[A_f^{\psi,\phi}[\lambda] = \left(\frac{1}{\lambda}\mathcal{O}_f[\psi,\phi]_{\lambda}/\mathcal{O}_f[\psi,\phi]_{\lambda}\right)(1-k)(\psi^{-1}\phi)\cong \mathbb{F}_{\lambda}(1-k)(\psi^{-1}\phi)\] and so we may view $c$ as a class in $H^1(\mathbb{Q},A_f^{\psi,\phi}[\lambda])$. The following short exact sequence of $\mathbb{F}_{\lambda}[G_{\mathbb{Q}}]$-modules:
\begin{equation*}
    0 \longrightarrow A_f^{\psi,\phi}[\lambda] \overset{i}\longrightarrow A_{f,\lambda}^{\psi,\phi} \overset{\lambda}\longrightarrow A_{f,\lambda}^{\psi,\phi} \longrightarrow 0
\end{equation*}
induces a long exact sequence in Galois cohomology, a piece of which is the following:
\begin{equation*}
    H^0(\mathbb{Q}, A_{f,\lambda}^{\psi,\phi}) \overset{\delta}\longrightarrow H^1(\mathbb{Q}, A_f^{\psi,\phi}[\lambda]) \overset{i_*}\longrightarrow H^1(\mathbb{Q}, A_{f,\lambda}^{\psi,\phi}).
\end{equation*}

\noindent Note that since $l> k+1$ we have that $(l-1)\nmid (k-1)$, and so $H^0(\mathbb{Q}, A_{f,\lambda}^{\psi,\phi})$ is trivial. It follows that $i_*$ is injective, so that we can lift $c$ to a non-trivial class $c' = i_*(c) \in H^1(\mathbb{Q}, A_{f,\lambda}^{\psi,\phi})$.

The aim is to show that $c'$ is a non-trivial element of a Bloch-Kato Selmer group, which we now define (as in \cite[\S 3]{bloch_kato_2007}). First let $B_{f,\lambda}^{\psi,\phi} = K_f[\psi,\phi]_{\lambda}(1-k)(\psi^{-1}\phi)$. For a prime $q \neq l$ we define:

    \[H^1_f(\mathbb{Q}_q, B_{f,\lambda}^{\psi,\phi}) = \text{ker}(H^1(D_q, B_{f,\lambda}^{\psi,\phi})
    \longrightarrow H^1(I_q, B_{f,\lambda}^{\psi,\phi})),\]

\noindent where $I_q\subset D_q\subset G_{\mathbb{Q}_q}$ are inertia and decomposition subgroups at $q$. The cohomology is taken with respect to continuous cocycles and coboundaries. Note also that the $f$ on the LHS is standard notation and is not related to the modular form $f$. When $q=l$ we define

    \[H^1_f(\mathbb{Q}_l, B_{f,\lambda}^{\psi,\phi}) = \text{ker}(H^1(D_l, B_{f,\lambda}^{\psi,\phi}) \longrightarrow H^1(I_l, B_{f,\lambda}^{\psi,\phi} \otimes_{\mathbb{Q}_l} B_{\text{crys}} )).\]

\noindent See \cite[\S 1]{bloch_kato_2007} for the definition of Fontaine's ring $B_{\text{crys}}$. The Bloch-Kato Selmer group of $B_{f,\lambda}^{\psi,\phi}$ is then $H^1_f(\mathbb{Q}, B_{f,\lambda}^{\psi,\phi})$, the subgroup of $H^1(\mathbb{Q}, B_{f,\lambda}^{\psi,\phi})$ consisting of classes which have local restriction lying in $H^1_f(\mathbb{Q}_q, B_{f,\lambda}^{\psi,\phi})$ for all primes $q$. 

Letting $\pi: B_{f,\lambda}^{\psi,\phi} \longrightarrow A_{f,\lambda}^{\psi,\phi}$ be the quotient map, we may also define the pushforward $H^1_f(\mathbb{Q}_q, A_{f,\lambda}^{\psi,\phi}) = \pi_*H^1_f(\mathbb{Q}_q,B_{f,\lambda}^{\psi,\phi})$. The Bloch-Kato Selmer group of $A_{f,\lambda}^{\psi,\phi}$ is then $H^1_f(\mathbb{Q}, A_{f,\lambda}^{\psi,\phi})$, the subgroup of $H^1(\mathbb{Q}, A_{f,\lambda}^{\psi,\phi})$ consisting of classes whose local restrictions lie in $H^1_f(\mathbb{Q}_q, A_{f,\lambda}^{\psi,\phi})$ for all primes $q$. Note that, since $l \nmid 2$ we may omit $q=\infty$. 

More generally, given a finite set of primes $\mathcal{P}$ with $l \notin \mathcal{P}$, we define $H^1_\mathcal{P}(\mathbb{Q}, A_{f,\lambda}^{\psi,\phi})$ to be the subgroup of  $H^1(\mathbb{Q}, A_{f,\lambda}^{\psi,\phi})$ consisting of classes whose local restrictions lie in $H^1_f(\mathbb{Q}_q, A_{f,\lambda}^{\psi,\phi})$ for all primes $q \notin \mathcal{P}$.

\begin{propn}\label{prop:BK}
    The congruence satisfied by $f$ gives the existence of a non-trivial element $c' \in H^1_{\mathcal{P}_{NM}}(\mathbb{Q}, A_{f,\lambda}^{\psi,\phi})$.
\end{propn}
\begin{proof}
    We have that $\rho_{f,\lambda}$ is unramified at each $q\nmid NMl$. It follows that restriction of $c'$ to $H^1(I_q, A_{f,\lambda}^{\psi,\phi})$ is 0 for such $q$. Then by \cite[Lemma 7.4]{brown_2007}, $c' \in H^1_f(\mathbb{Q}_q, A_{f,\lambda}^{\psi,\phi})$. Under the assumption that $l > k+1$, the representation $\rho_{f,\lambda}$ is crystalline at $l$ and we deduce that $c' \in H^1_f(\mathbb{Q}_l, A_{f,\lambda}^{\psi,\phi})$, as a consequence of \cite[Proposition 2.2]{diamond_flach_guo_2004}. Since the necessary local conditions are satisfied, we have that $c' \in H^1_{\mathcal{P}_{NM}}(\mathbb{Q}, A_{f,\lambda}^{\psi,\phi}))$.
\end{proof}

Let $C_{k,l}^{\psi,\phi} = (\mathbb{Q}_l/\mathbb{Z}_l)(1-k)(\psi^{-1}\phi)$. Note that since $\lambda|l$, the module $A_{f,\lambda}^{\psi,\phi}$ decomposes as a direct sum of copies of $C_{k,l}^{\psi,\phi}$. Proposition \ref{prop:BK} then implies the existence of a non-trivial element $c'\in H^1_{\mathcal{P}_{NM}}(\mathbb{Q},C_{k,l}^{\psi,\phi})$ by projection. We now discuss how the existence of such an element agrees with the Bloch-Kato conjecture.

Consider the partial Dirichlet $L$-value $L_{\mathcal{P}_{NM}}(k, \psi\phi^{-1})$, i.e. with Euler factors at primes $q \in \mathcal{P}_{NM}$ omitted. Letting $\lambda'$ be as in Condition (1) of Conjecture \ref{conj:general}, below is a reformulation of a special case of the $\lambda'$-part of the Bloch-Kato conjecture (as in \cite{diamond_flach_guo_2004}, proved in this case by Huber and Kings in \cite{huber_kings_2003}).

\begin{conj}
\label{conj:BK}
\[\text{ord}_{\lambda'}\left( \frac{L_{\mathcal{P}_{NM}}(k, \psi\phi^{-1})}{g(\psi\phi^{-1})(2\pi i)^k}\right)
   = \text{ord}_{\lambda'} \left( \frac{\text{Tam}^0_\lambda(C_{k,l}^{\psi,\phi}) \#H^1_{\mathcal{P}_{NM}}(\mathbb{Q},C_{k,l}^{\psi,\phi})}{\#H^0(\mathbb{Q},C_{k,l}^{\psi,\phi})} \right).\]
\end{conj}
 \noindent We omit the definition of the Tamagawa factor $\text{Tam}^0_\lambda(C_{k,l}^{\psi,\phi})$, but note that it is trivial in this case since $l > k+1$ and $\lambda' \mid l$, by \cite[Theorem 4.1.1(iii)]{bloch_kato_2007}. We also know that $H^0(\mathbb{Q},C_{k,l}^{\psi,\phi}))$ is trivial, and so \[\text{ord}_{\lambda'}\left( \frac{L_{\mathcal{P}_{NM}}(k, \psi\phi^{-1})}{g(\psi\phi^{-1})(2\pi i)^k}\right)=\text{ord}_{\lambda'}(\#H^1_{\mathcal{P}_{NM}}(\mathbb{Q},C_{k,l}^{\psi,\phi})).\] Hence if we can show that $\lambda'$ divides the partial $L$-value then we know there is a non-trivial element in the Bloch-Kato Selmer group.

\begin{propn}
\label{propn:divideL}
     Condition (1) of Conjecture \ref{conj:general} implies the condition \[\text{ord}_{\lambda'}\left( \frac{L_{\mathcal{P}_{NM}}(k, \psi\phi^{-1})}{g(\psi\phi^{-1})(2\pi i)^k}\right)>0.\]
\end{propn}

\begin{proof}
    By the functional equation for $L(s,\psi\phi^{-1})$ we have: 
    \begin{align*}
        \frac{L_{\mathcal{P}_{NM}}(k, \psi\phi^{-1})}{g(\psi\phi^{-1})(2\pi i)^k} &= \frac{(-1)^{|\mathcal{P}_{NM}|}}{\phi(NM)(NM)^k}\frac{L(k, \psi \phi^{-1})}{ g(\psi\phi^{-1})(2\pi i)^k}\prod_{p \in \mathcal{P}_{NM}} (\psi(p)-\phi(p)p^k)\\ &=\frac{(-1)^{|\mathcal{P}_m|+k}}{2(k-1)!\phi(NM)(N^2M)^k}L(1-k,\psi^{-1}\phi)\prod_{p \in \mathcal{P}_{NM}} (\psi(p)-\phi(p)p^k)
    \end{align*}
    The claim follows, since $l\nmid NM$ and $l>k+1$.
\end{proof}
The above shows that Condition (1) of Conjecture \ref{conj:general} provides a non-trivial element $c'\in H^1_{\mathcal{P}_{NM}}(\mathbb{Q},C_{k,l}^{\psi,\phi})$, and that this is in line with the Bloch-Kato Conjecture. Since Condition (2) is (conjecturally) telling us that a newform $f$ of level $NM$ can be found to satisfy the congruence, we might naively expect that the corresponding element $c'\in H^1_{
\mathcal{P}_{NM}}(\mathbb{Q},C_{k,l}^{\psi,\phi})$ is ``new", i.e.\ that $c'\notin H^1_{\mathcal{P}_{Nd}}(\mathbb{Q},C_{k,l}^{\psi,\phi})$ for each $d| M$ with $d\neq M$. However, this may not be the case, since considering the Bloch-Kato quotient for such a divisor $d$ gives: \begin{align*}
    \text{ord}_{\lambda^\prime}\left( \frac{ \#H^1_{\mathcal{P}_{NM}}(\mathbb{Q},C_{k,l}^{\psi,\phi})}{\#H^1_{\mathcal{P}_{Nd}}(\mathbb{Q},C_{k,l}^{\psi,\phi})}\right) &= \text{ord}_{\lambda^\prime}\left( \frac{L_{\mathcal{P}_{NM}}(k, \psi\phi^{-1})}{L_{\mathcal{P}_{Nd}}(k, \psi\phi^{-1})} \right)\\ &=\text{ord}_{\lambda^\prime}\left( \frac{\prod_{p \in \mathcal{P}_{M/d}}(\psi(p)-\phi(p)p^k)}{\phi(M/d)(M/d)^k}\right),
\end{align*}
revealing that new elements can only be accounted for by local divisibility conditions of the form $\text{ord}_{\lambda'}(\psi(p)-\phi(p)p^k)>0$ (as in Case (A) in the proof of Theorem \ref{thm:reverse}). It follows that the primes $p\in\mathcal{P}_M$ satisfying $\text{ord}_{\lambda'}(\psi(p)-\phi(p)p^{k-2})>0$ and $\text{ord}_{\lambda'}(\psi(p)-\phi(p)p^k)=0$ in Condition (2) cannot contribute towards new elements.
 
\begin{propn}\label{prop:BK1}
    Let $d|M$ and $d\neq M$. Then there exists a prime $p\in \mathcal{P}_{M/d}$ such that $\text{ord}_{\lambda'}(\psi(p)-\phi(p)p^k)>0$ if and only if \[\text{ord}_{\lambda'}\left(\frac{\#H^1_{\mathcal{P}_{NM}}(\mathbb{Q},C_{k,l}^{\psi,\phi})}{\#H^1_{\mathcal{P}_{Nd}}(\mathbb{Q},C_{k,l}^{\psi,\phi})}\right) > 0.\] In particular, $\text{ord}_{\lambda'}(\psi(p)-\phi(p)p^k)>0$ if and only if there exists an element $c'_p\in H^1_{\mathcal{P}_{NM}}(\mathbb{Q},C_{k,l}^{\psi,\phi})$ that is ``$p$-new", i.e.\ $c'_p\notin H^1_{\mathcal{P}_{NM/p}}(\mathbb{Q},C_{k,l}^{\psi,\phi})$.
    \end{propn}

    \begin{proof}
    This follows from the above discussion.
    \end{proof}

    \noindent The above should be roughly compared with the situation in Theorem \ref{thm:prime} and the proceeding discussion, where Conditions (1) and (2) alone were only able to guarantee the existence of $p$-newforms satisfying the congruence for each $p\in\mathcal{P}_M$, as opposed to a genuine newform. Based on this, we conjecture the following.

    \begin{conj}\label{conj:pnew}
    Let $\mathcal{S} = \{p\in \mathcal{P}_M\,|\,\text{ord}_{\lambda'}(\psi(p)-\phi(p)p^k)>0\}$. Then the class $c'\in H^1_{\mathcal{P}_{NM}}(\mathbb{Q},C_{k,l}^{\psi,\phi})$ provided by Proposition \ref{prop:BK} is ``$p$-new" for each $p\in \mathcal{S}$.
    \end{conj}

   Now suppose that $\text{ord}_{\lambda'}(\psi(p)-\phi(p)p^k) = 0$ for every $p\in\mathcal{P}_{M/d}$. In order to satisfy Condition (2) we would then have the following conditions for each $p\in\mathcal{P}_{M/d}$: \begin{align*}&\text{ord}_{\lambda'}(\psi(p)-\phi(p)p^{k-2})>0,\\  &\text{ord}_{\lambda'}(L(1-k,\psi^{-1}\phi))>0.\end{align*} Here, the situation is not fully clear (at least not to the authors). One conclusion that can be made in this case is the following.

    \begin{propn}\label{prop:BK2}
    Let $d|M$ and $d\neq M$. If $\text{ord}_{\lambda'}(\psi(p)-\phi(p)p^{k-2})>0$ for each $p\in\mathcal{P}_{M/d}$ then \[\text{ord}_{\lambda'}\left(\frac{\#H^1_{\mathcal{P}_{NM}}(\mathbb{Q},C_{k-2,l}^{\psi,\phi})}{\#H^1_{\mathcal{P}_{Nd}}(\mathbb{Q},C_{k-2,l}^{\psi,\phi})}\right) \geq \#\mathcal{P}_{M/d}.\]
    
    In particular, if $\text{ord}_{\lambda'}(\psi(p)-\phi(p)p^{k-2})>0$ then there exists an element $c_p{''}\in H^1_{\mathcal{P}_{NM}}(\mathbb{Q},C_{k-2,l}^{\psi,\phi})$ that is ``$p$-new", i.e.\ $c_p{''}\notin H^1_{\mathcal{P}_{NM/p}}(\mathbb{Q},C_{k-2,l}^{\psi,\phi})$.
    \end{propn}

    \begin{proof}
    This follows by considering the Bloch-Kato quotient of weight $k-2$: \begin{align*}\text{ord}_{\lambda^\prime}\left( \frac{ \#H^1_{\mathcal{P}_{NM}}(\mathbb{Q},C_{k-2,l}^{\psi,\phi})}{\#H^1_{\mathcal{P}_{Nd}}(\mathbb{Q},C_{k-2,l}^{\psi,\phi})}\right) &= \text{ord}_{\lambda^\prime}\left( \frac{L_{\mathcal{P}_{NM}}(k-2, \psi\phi^{-1})}{L_{\mathcal{P}_{Nd}}(k-2, \psi\phi^{-1})} \right)\\ &=\text{ord}_{\lambda^\prime}\left( \frac{\prod_{p \in \mathcal{P}_{M/d}}(\psi(p)-\phi(p)p^{k-2})}{\phi(M/d)(M/d)^{k-2}}\right).\end{align*}
    \end{proof}

    Alternatively, the divisibility condition in the above Proposition implies Conditions (1) and (2) for weight $k-2$. By Theorem \ref{thm:eigenform} there exists a congruence between the eigenvalues of an eigenform $g\in S_{k-2}(\Gamma_0(NM/d),\tilde{\chi})$ and the eigenvalues of $E_{k-2}^{\psi,\phi}$ (at ``good" primes). By a similar argument to earlier, this supplies a non-trivial element $c^{''}\in H^1_{\mathcal{P}_{NM/d}}(\mathbb{Q},C_{k-2,l}^{\psi,\phi})$. Conjecture \ref{conj:general} would let us take $g$ to be a newform, and so we are led to conjecture the following.

    \begin{conj}\label{conj:pnew2}
    Let $\mathcal{S} = \{p\in \mathcal{P}_M\,|\,\text{ord}_{\lambda'}(\psi(p)-\phi(p)p^{k-2})>0\}$. Then the class $c^{''}\in H^1_{\mathcal{P}_{NM/d}}(\mathbb{Q},C_{k-2,l}^{\psi,\phi})$ constructed above is $p$-new for each $p\in \mathcal{S}$.
    \end{conj}

    \noindent In summary, we have considered the two different ways in which Conditions (1) and (2) can hold. By Bloch-Kato, each implies the existence of a collection ``$p$-new" elements in a certain Bloch-Kato Selmer group (i.e.\ Propositions \ref{prop:BK1} and \ref{prop:BK2}). Further, we expect that both collections of $p$-new elements should be explained by the existence of two ``new" elements, each arising from a newform congruence implied by Conjecture \ref{conj:general} (i.e.\ Conjectures \ref{conj:pnew} and \ref{conj:pnew2}).  In the case of $M=p$ prime, we remark that Theorem \ref{thm:prime} implies the truth of these conjectures, but for more general square-free $M$, little seems to be known.

\section{Examples}
\label{section:eg}

Below, we give brief computational examples of Conjecture \ref{conj:general}, using data provided by the LMFDB database \cite{lmfdb}. First we consider examples demonstrating Theorem \ref{thm:prime}. In each of the examples below, the prime $\lambda$ satisfies the conditions of Case (A) in the proof of Theorem \ref{thm:reverse}.

\begin{eg}
Take $N=5$, $M=2$ and $k=8$. Also let $\psi = \mathds{1}$ and $\phi = \left(\frac{\cdot}{5}\right)$ (so that $\chi = \phi$). The only prime $\lambda'$ of $\mathbb{Z}[\psi,\phi] = \mathbb{Z}$ satisfying the conditions of Theorem \ref{thm:prime} is $\lambda' = 257$. 

\noindent Indeed, the newform $f \in S_8^{\text{new}}(\Gamma_0(10), \tilde{\chi})$ with LMFDB label [10.8.b.a] satisfies the congruence \[a_q(f) \equiv 1+\left(\frac{q}{5}\right)q^{7} \bmod \lambda\] for all $q\neq 2,5$ and for some fixed prime $\lambda\mid \lambda'$ of $\mathcal{O}_f[\psi, \phi] = \mathcal{O}_f$ (the ring of integers of the quartic field generated by a root of $x^4-15x^2+64$).
\end{eg}

\begin{eg}
Take $N=7$, $M=2$ and $k=7$. Also let $\psi = \mathds{1}$ and $\phi$ be the primitive mod $7$ character satisfying $\phi(3) = \zeta_6$ (so that $\chi = \phi$). The only prime $\lambda'$ of $\mathbb{Z}[\psi,\phi] = \mathbb{Z}[\zeta_6]$ satisfying the conditions of Theorem \ref{thm:eigenform} is $\lambda' = \langle 337, \zeta_6+128\rangle$ (lying above $l = 337$). 

\noindent Indeed, the newform $f \in S_7^{\text{new}}(\Gamma_0(14), \tilde{\chi})$ with LMFDB label [14.7.d.a] satisfies the congruence \[a_q(f) \equiv 1 + \phi(q)q^6 \bmod \lambda\] for all $q\neq 2,7$ and for some fixed prime $\lambda\mid \lambda'$ of $\mathcal{O}_f[\psi,\phi] = \mathcal{O}_f[\zeta_6]$ (here $\mathcal{O}_f$ is the ring of integers of a degree $8$ number field).
\end{eg}

We finish with an example demonstrating Conjecture \ref{conj:general} in a case where $M$ is composite.

\begin{eg}
Take $N=7, M=6$ and $k=6$. Let $\psi = \mathds{1}$ and $\phi$ be the primitive mod $7$ character satisfying $\phi(3) = \zeta_3^2$ (so that $\chi = \phi$). The only prime $\lambda'$ of $\mathbb{Z}[\psi,\phi] = \mathbb{Z}[\zeta_6]$ satisfying the conditions of Theorem \ref{thm:eigenform} is $\lambda' = \langle 73, \zeta_6+64\rangle$ (lying above $l = 73$). 

\noindent Indeed, the newform $f\in S_6^{\text{new}}(\Gamma_0(42),\tilde{\chi})$ with LMFDB label [42.6.e.c] satisfies the congruence: \[a_q(f) \equiv 1 + \phi(q)q^6\bmod \lambda\] for all primes $q\neq 2,7$ and for a fixed prime $\lambda\mid \lambda'$ of $\mathcal{O}_{f}[\psi,\phi] = \mathcal{O}_{f}[\zeta_6]$ (here $\mathcal{O}_f$ is the ring of integers of a degree $4$ number field).
\end{eg}

\bibliography{References}
\bibliographystyle{plain}

\end{document}